\documentclass{article}
\usepackage{amsfonts}
\usepackage{amsmath}
\usepackage{amssymb}
\usepackage{amsthm}
\usepackage{tikz}
\usepackage{todonotes}

\newtheorem{theorem}{Theorem}

\newtheorem{claim}[theorem]{Claim}

\newtheorem{conjecture}[theorem]{Conjecture}
\newtheorem{corollary}[theorem]{Corollary}

\newtheorem{example}[theorem]{Example}

\newtheorem{lemma}[theorem]{Lemma}

\newtheorem{problem}[theorem]{Problem}

\newtheorem{remark}[theorem]{Remark}

\begin{document}

\title{Polynomial Method in Tilings}
\author{Peter Horak$^{1}$ \ \ Dongryul Kim$^{2}$ \\
\normalsize$^{1}$ School of Interdisciplinary Arts \& Sciences \\
\normalsize University of Washington, Tacoma, WA \\
\normalsize e-mail: horak@uw.edu \\
\normalsize $^{2}$ Harvard University, Cambridge, MA \\
\normalsize e-mail: dkim04@college.harvard.edu }
\maketitle

\begin{abstract}
\noindent In this paper we introduce a new algebraic method in tilings. Combining this method with Hilbert's Nullstellensatz we obtain a necessary condition for tiling $n$-space by translates of a cluster of cubes. Further, the polynomial method will enable us to show that if there exists a tiling of $n$-space by translates of a cluster $V$ of prime size then there is a lattice tiling by $V$ as well. Finally, we provide supporting evidence for a conjecture that each tiling by translates of a prime size cluster $V$ is lattice if $V$ generates $n$-space.
\end{abstract}

\section{Introduction}

A cluster in $\mathbb{R}^{n}$ is the union of unit cubes centered at integer points with their sides parallel to coordinate axis; we note that a cluster does not have to be connected. This paper is devoted to tilings of $\mathbb{R}^{n}$ by translates of a cluster.\bigskip 

An interest in tilings of $\mathbb{R}^{n}$ by cubes goes back to a conjecture raised by Minkowski \cite{Minkowski} in 1904; the conjecture stemmed from his work on geometry of numbers and quadratic forms.

\begin{conjecture}[Minkowski] 
Each lattice tiling of $\mathbb{R}^{n}$ by cubes contains twins, a pair of cubes that share whole $n-1$ dimensional face. 
\end{conjecture}

\noindent Minkowski's conjecture was settled in the affirmative in 1941 by Haj\'{o}s \cite{Hajos} who introduced in that paper a powerful algebraic method called ``\emph{splitting of groups}." We note that although a cluster is a very special type of a tile, it provides a simplest known counterexample to part (b) of the 18th problem of Hilbert:

\begin{problem}
If congruent copies of a polyhedron $P$ tile $\mathbb{R}^{3}$, is there a group of motions so that copies of $P$ under this group tile $\mathbb{R}^{3}$?
\end{problem}

\noindent In other words, the second part of the problem asks whether there exists a polyhedron, which tiles 3-dimensional Euclidean space but does not admit an isohedral (tile-transitive) tiling. It is shown in \cite{BE} that there is a periodic tiling of $\mathbb{R}^{2}$ by a cluster depicted in Fig.1, but no isohedral tiling of $\mathbb{R}^{2}$ exists.

\begin{figure}[h]
\centering
\begin{tikzpicture}[scale=0.8]
\draw (0,0) -- (2,0) -- (2,1) -- (3,1) -- (3,0) -- (5,0) -- (5,1) -- (4,1) -- (4,2) -- (1,2) -- (1,1) -- (0,1) -- cycle;
\draw (1,0) -- (1,1) -- (2,1) -- (2,2);
\draw (3,2) -- (3,1) -- (4,1) -- (4,0);
\end{tikzpicture} \\
\vspace{1em}
Fig.1.
\end{figure}
%

In this paper we deal only with face-to-face (=regular) tilings of $\mathbb{R}^{n}$ by a cluster $C$. It is not difficult to see that such a tilings can be seen as a tiling of $\mathbb{Z}^{n}$ by translates of a subset $V$ comprising centers of cubes in $C$. Thus, from now on, by a tile we will mean a set $V\subset \mathbb{Z}^{n}$. Throughout the paper we assume that $0 \in V$, and we deal exclusively with tilings $\mathcal{T}$ of $\mathbb{Z}^{n}$ by \emph{translates} of $V$; i.e., 
\begin{equation*}
\mathcal{T}=\{V+l ; l \in \mathcal{L}\}.
\end{equation*}
\noindent As $0 \in V$, we will identify each tile $V + l$ in $\mathcal{T}$ with $l$. A tiling $\mathcal{T}$ is termed periodic (lattice), if $\mathcal{L}$ is periodic (lattice). Since $\mathbb{Z}^{n}$ is a group, the fact that $V$ tiles $\mathbb{Z}^{n}$ can be expressed as 
\begin{equation*}
\mathbb{Z}^{n}=V+\mathcal{L},
\end{equation*}
meaning that each element of $\mathbb{Z}^{n}$ can be written in a unique way as the sum of an element in $V$ and an element of $\mathcal{L}$, and also as
\begin{equation*}
\lvert (-V+x) \cap \mathcal{L} \rvert = 1
\end{equation*}
for each $x \in \mathbb{Z}^{n}$. In the area of tilings of $\mathbb{Z}^{n}$ by translates of a set $V$ most research is oriented towards solving several long-standing conjectures.

\begin{conjecture}[Lagarias-Wang 1996, \cite{LW}]
If $V$ tiles $\mathbb{Z}^{n}$, then $V$ admits a periodic tiling.
\end{conjecture}

\noindent It is easy to see that the conjecture is true in the 1-dimensional case, but it is still open even for $n=2$. It is known though that, for $n=2$, the conjecture is true for polyominoes, cf. \cite{BE}, i.e., if the corresponding cluster of cubes in $\mathbb{R}^{2}$ is connected. Moreover, Szegedy \cite{Szegedy} proved the conjecture in the case when $V$ is of a prime size. Further, Nivat \cite{NIVAT} conjectured, that if $V$ satisfies a complexity assumption, then each tiling of $\mathbb{Z}^{2}$ by $V$ is periodic. We note that the famous Keller's conjecture \cite{Keller} saying that each tiling of $\mathbb{R}^{n}$ by cubes contains a pair of twin cubes was proved to be false for all $n \ge 8$, but it is still open for $n=7$. \bigskip 

Our research has been motivated by two conjectures stated below. The first of them is likely the most famous conjecture in the area of error-correcting Lee codes:

\begin{conjecture}[Golomb-Welch 1969, \cite{GW}] The Lee sphere 
\begin{equation*}
S_{n,r} = \{ \mathbf{x} \in \mathbb{Z}^{n} : \lvert x_{1} \rvert + \dots +\lvert x_{n} \rvert \leq r \}
\end{equation*}
does not tile $\mathbb{Z}^{n}$ for $n \geq 3$ and $r \geq 2$.
\end{conjecture}

\noindent Although there is a sizable literature on the topic, the
conjecture is far from being solved. \bigskip

The $n$-cross is a cluster in $\mathbb{R}^{n}$ comprising $2n+1$ cubes, a central one and its reflections in all faces. Thus, $\{0,\pm e_{1}, \dots ,\pm e_{n}\}$ is the set of centers of cubes in the $n$-cross in $\mathbb{Z}^{n}$. It is known, see \cite{H3}, that if $2n+1$ is not a prime then there are uncountably many non-congruent tilings of $\mathbb{Z}^{n}$ by the $n$-cross. It was conjectured there that:

\begin{conjecture}
\label{Cross}If $2n+1$ is a prime then, up to a congruence, there is only
one tiling of $\mathbb{Z}^{n}$ by $n$-cross.
\end{conjecture}

\noindent We believe, if true, the conjecture goes against our intuition that says: The higher the dimension, the more freedom we get. The conjecture has been proved for $n=2,3$ in \cite{H3} and for $n=5$ in \cite{HH}. Thus, there is a unique tiling of $\mathbb{Z}^{n}$ by crosses for $n=2,3$, there are uncountably many tilings of $\mathbb{Z}^{4}$ by crosses, but in $\mathbb{Z}^{5}$ there is again a unique tiling by crosses. \bigskip 

To attack these two conjecture we first describe a new algebraic method, so-called ``polynomial method" that will enable us to prove some general results on tiling $\mathbb{Z}^{n}$ by translates of a cluster. We note that a similar method has been independently developed and used in \cite{KS}, where the authors focus on Nivat's conjecture. Szegedy \cite{Szegedy} proved, using a new algebraic technique based on quasigroups, that if a tile $V$ is of a prime size then each tiling of $\mathbb{Z}^{n}$ by translates of $V$ is periodic. The polynomial method provides a different proof of this result:

\begin{theorem}
Let $V \subset \mathbb{Z}^{n}$, and $\mathcal{T}$ be a tiling of $\mathbb{Z}^{n}$ by translates of $V$. If $\lvert V \rvert =q$ is prime, then $q(\mathbf{v}-\mathbf{w})$ is a period of $\mathcal{T}$ for any $\mathbf{v}, \mathbf{w}\in V$. 
\end{theorem}

\noindent Further, applying Hilbert Nullstellensatz, we provide a necessary condition for the existence of a tiling $\mathbb{Z}^{n}$ by translates of a generic (arbitrary) set $V$. With this in hand we prove that if $V = \{ 0, v_{1} , \dots , v_{q-1} \}$ is of a prime size $q$ and $\{v_{1} , \dots ,v_{q-1}\}$ generate $\mathbb{Z}^{n}$ then there is a tiling of $\mathbb{Z}^{n}$ by translates of $V$ if and only if there is a lattice tiling of $\mathbb{Z}^{n}$ by $V$. We conjecture a much stronger result: 

\begin{conjecture}
\label{Lattice} Let $V=\{0,\mathbf{v}_{1},\dots ,\mathbf{v}_{q-1}\}\subset \mathbb{Z}^{n}$ of a prime size $q$ tiles $\mathbb{Z}^{n}$ by translates, and $\{\mathbf{v}_{1},\dots ,\mathbf{v}_{q-1}\}$ generate $\mathbb{Z}^{n}$. Then there is a unique tiling, up to a congruency, of $\mathbb{Z}^{n}$ by $V$ and this tiling is lattice.
\end{conjecture}

\noindent Clearly, if true, the above conjecture would imply Conjecture~\ref{Cross}. To provide supporting evidence we prove the above conjecture for all primes $\leq 7$. \bigskip 


\section{Polynomial Method}

First we describe \emph{Polynomial Method} that represents our main tool when tackling various tilings problems. Then we state results that, in our opinion, are of interest on their own, but also constitute an important ingredient in the proofs of main theorems of this paper. \bigskip 

Let $\mathcal{T}=\{V+l;l\in \mathcal{L\}}$ be a tiling of $\mathbb{Z}^{n}$ by translates of $V$. We define a linear map $T_{\mathcal{T}} : \mathbb{Z}[x_{1}^{\pm 1}, \dots , x_{n}^{\pm 1}]\rightarrow \mathbb{Z}$, where $\mathbb{Z}[x_{1}^{\pm 1}, \dots ,x_{n}^{\pm 1}]$ is the commutative ring of Laurent polynomials generated by $x_{1}^{\pm 1}, \dots , x_{n}^{\pm 1}$, such that, for every $(a_{1}, \dots ,a_{n})\in \mathbb{Z}^{n}$, 

\begin{equation*}
T_{\mathcal{T}} (x_{1}^{a_{1}}\cdots x_{n}^{a_{n}}) =%
\begin{cases}
1 & \text{if }(a_{1},\cdots ,a_{n})\in \mathcal{L} \\ 
0 & \text{otherwise.}%
\end{cases}%
\end{equation*}

\noindent If the tiling $\mathcal{T}$ will be clear from the context we will drop the subscript and write simply $T$. We note that $T$ is uniquely determined as the monomials $x_{1}^{a_{1}}\cdots x_{n}^{a_{n}}$ form a basis of the ring. Let $Q_{V}\in \mathbb{Z}[x_{1}^{\pm 1}, \dots , x_{n}^{\pm 1}]$ be a polynomial associated with $V$, where 
\begin{equation*}
Q_{V}(x_{1}, \dots ,x_{n})=\sum_{(a_{1}, \dots ,a_{n}) \in (-V)} x_{1}^{a_{1}} \cdots x_{n}^{a_{n}}.
\end{equation*}
Then for any monomial $x_1^{m_1} \cdots x_n^{m_n}$, 
\begin{align*}
T(x_1^{m_1} \cdots x_n^{m_n} Q_V) &= \sum_{(a_1, \dots, a_n) \in (-V)} \lvert \{(a_1 + m_1, \dots, a_n + m_n)\} \cap \mathcal{L} \rvert \\
&= \lvert (-V + (m_1, \dots, m_n)) \cap \mathcal{L} \rvert = 1.
\end{align*}
Since the map $T$ is linear and any polynomial is a linear combination of monomials, we can immediately extend this equality to
\[ T( P Q_V) = P(1, \dots, 1) \]
for any polynomial $P \in \mathbb{Z}[x_1^{\pm 1}, \dots, x_n^{\pm 1}]$. \bigskip

In what follows we will present results on tilings of $\mathbb{Z}^{n}$ by translates of a set $V\subset \mathbb{Z}^{n}$. Most of these results will be proved by utilizing properties of the linear map $T$ and the polynomial $Q_{V}$. We have termed this approach \emph{Polynomial Method}. \bigskip

We start with a technical statement:

\begin{theorem} \label{C} 
Let $\mathcal{T}$ be a tiling of $\mathbb{Z}^{n}$ by translates of $V$, and let $a$ be an integer relatively prime to $\lvert V \rvert$. Then, for any polynomial $P \in \mathbb{Z}[x_{1}^{\pm 1},\dots, x_{n}^{\pm 1}]$, we have 
\begin{equation*}
T(PQ_{V}(x_{1}^{a},\dots ,x_{n}^{a}))=P(1,\dots ,1).
\end{equation*}
\end{theorem}

\begin{proof}
This statement follows directly from the two lemmas given below since $a$
can be represented as a product of primes not dividing $\lvert
V\rvert $ and possibly $-1$.
\end{proof}

\begin{lemma} \label{A}
Let $p=1$, or $p$ be a prime which does not divide $\lvert V \rvert$. Then 
\begin{equation*}
T(PQ_{V}(x_{1}^{p},\dots ,x_{n}^{p}))=P(1,\dots ,1)
\end{equation*}
for any polynomial $P\in \mathbb{Z}[x_{1}^{\pm 1},\dots ,x_{n}^{\pm 1}]$.
\end{lemma}

\begin{proof}
Since the map $T$ is linear, it is sufficient to prove that $T(MQ(x_{1}^{p},\dots ,x_{n}^{p}))=1$ for any monomial $M$. We have
\begin{align*}
T(MQ_{V}(x_{1}^{p},\dots ,x_{n}^{p})) & \equiv T(MQ_{V}^{p})=T(MQ_{V}^{p-1}Q_{V})\\
&=(Q_{V}(1, \dots ,1))^{p-1}=\lvert V \rvert
^{p-1}\equiv 1 \pmod{p}
\end{align*}
\noindent since $T(RQ_{V})=R(1, \dots, 1)$ for any polynomial $R$. Thus $T(MQ_{V}(x_{1}^{p},\dots ,x_{n}^{p}))\geq 1$ for all monomials $M$. \bigskip 

We also have 
\begin{align}
T(MQ_{V}(x_{1}^{p},\dots ,x_{n}^{p})Q_{V}) & = \sum_{\mathbf{v}\in V} T (M \cdot x_{1}^{v_{1}}\cdots x_{n}^{v_{n}} \cdot Q_{V}(x_{1}^{p},\dots ,x_{n}^{p})) \notag \\
& \geq \sum_{\mathbf{v}\in V} 1 = \lvert V \rvert  \label{eqn:ineq1}
\end{align}
while on the other hand, 
\begin{equation*}
T(MQ_{V}(x_{1}^{p},\dots ,x_{n}^{p})Q_{V}) = Q_{V}(1^{p},\dots ,1^{p})=\lvert V \rvert.
\end{equation*}
It follows that the equality holds for every term in (\ref{eqn:ineq1}). For some fixed $\mathbf{v}\in V$, we have $T(M\cdot x_{1}^{v_{1}} \cdots x_{n}^{v_{n}}\cdot Q(x_{1}^{p},\dots ,x_{n}^{p}))=1$ for every monomial $M$. Therefore $T(MQ(x_{1}^{p},\dots ,x_{n}^{p}))=1$ for every monomial $M$.
\end{proof}

\begin{lemma} \label{B} 
\begin{equation*}
T(PQ_{V}(x_{1}^{-1},\dots ,x_{n}^{-1})) = P(1,\dots ,1)
\end{equation*}
for any polynomial $P\in \mathbb{Z}[x_{1}^{\pm 1},\dots ,x_{n}^{\pm 1}]$.
\end{lemma}

\begin{proof}
Again, it is sufficient to prove it for monomials. We first prove 
\begin{equation*}
T(MQ_{V}(x_{1}^{-1},\dots ,x_{n}^{-1})) \leq 1
\end{equation*}
for any monomial $M$. Suppose that 
\begin{equation*}
T(Mx_{1}^{-v_{1}}\cdots x_{n}^{-v_{n}})=T(Mx_{1}^{-u_{1}} \cdots x_{n}^{-u_{n}})=1
\end{equation*}
for some distinct $\mathbf{v},\mathbf{u}\in (-V)$. Then letting $M^{\prime}=Mx_{1}^{-v_{1}-u_{1}}\cdots x_{n}^{-v_{n}-u_{n}}$, we get 
\begin{equation*}
T(M^{\prime }Q_{V})\geq T(M^{\prime }x_{1}^{v_{1}} \cdots x_{n}^{v_{n}})+T(M^{\prime }x_{1}^{u_{1}}\cdots x_{n}^{u_{n}}) = 2
\end{equation*}
which contradicts the original property of $Q_{V}$. Thus $T(MQ_{V}(x_{1}^{-1},\dots ,x_{n}^{-1}))\leq 1$ for all $M$. \bigskip

Consider $MQ_{V}(x_{1}^{-1},\dots ,x_{n}^{-1})Q_{V}$. Because $T(MQ_{V}(x_{1}^{-1},\dots ,x_{n}^{-1})Q_{V})=Q_{V}(1,\dots ,1)=\lvert V \rvert$ and 
\begin{equation*}
T(MQ_{V}(x_{1}^{-1},\dots ,x_{n}^{-1})Q_{V})\leq \sum_{\mathbf{v}\in V} 1 = \lvert V \rvert,
\end{equation*}
all terms must attain equality. It follows that $T(MQ_{V}(x_{1}^{-1},\dots, x_{n}^{-1}))=1$ for any monomial $M$.
\end{proof}

\bigskip

\begin{corollary} \label{blowout}
Let $\mathcal{T} = \{V+l;l\in \mathcal{L\}}$ be a tiling of $\mathbb{Z}^{n}$ by translates of $V$, and let $a$ be an integer relatively prime to $\lvert V \rvert $ or $a=-1$. Then $\mathcal{T}_{a} = \{aV+l; l \in \mathcal{L\}}$ is a tiling of $\mathbb{Z}^{n}$ by translates of a "blow-up" tile $aV=\{av;v\in V\}$.
\end{corollary}

\begin{proof}
Set $S = aV$. Then 
\begin{equation*}
Q_{S} (x_{1}, \dots ,x_{n}) = \sum_{(v_{1}, \dots ,v_{n}) \in (-V)} x_{1}^{av_{1}}\cdots x_{n}^{av_{n}} =Q_{V}(x_{1}^{a}, \dots ,x_{n}^{a}).
\end{equation*}
By the above theorem, 
\begin{equation*}
T(MQ_S) = T(MQ_{V}(x_{1}^{a},\dots ,x_{n}^{a}))=M(1,\dots ,1)=1
\end{equation*}
for any monomial $M$. Thus, for any $x \in \mathbb{Z}^{n}$, 
\begin{equation*}
\lvert (-S+x)\cap \mathcal{L}\rvert =1,
\end{equation*}
that is, $\mathcal{T}_{a} =\{aV+l;l\in \mathcal{L\}}$ is a tiling of $\mathbb{Z}^{n}$ by translates of $aV$.
\end{proof}

\bigskip

The following corollary can be found in \cite{Szegedy}. We provide
here a short proof of this result.

\begin{corollary} \label{CC}
Let $\mathcal{T} = \{V+l;l\in \mathcal{L}\}$ be a tiling of $\mathbb{Z}^{n}$ by translates of $V$, and let $a$ be an integer relatively prime to $\lvert V \rvert$. Then $l+a(v-w)\notin \mathcal{L}$ for each $l \in \mathcal{L}$ and $v,w \in V$.
\end{corollary}

\begin{proof}
By Corollary~\ref{blowout}, $\mathcal{T}_{a}$ $=\{aV+l;l\in \mathcal{L\}}$
is a tiling of $\mathbb{Z}^{n}$ by translates of $aV$, hence $\mathbb{Z}^{n} = aV + \mathcal{L}$. Assume that $l+a(v-w)\in \mathcal{L}$. Then 
\begin{align*}
l+av &= aw+[l+a(v-w)] \quad \text{but also} \\
l+av &= av+l;
\end{align*}
that is, $l+av \in \mathbb{Z}^{n}$ would be covered by two distinct tiles of $\mathcal{T}_{a}$.
\end{proof}

\section{A Necessary Condition for the Existence of a Tiling}

The main goal of this section is to present a necessary condition for the existence of a tiling of $\mathbb{Z}^{n}$ by translates of a generic (arbitrary) tile $V$. To the best of our knowledge this is the first condition of its type. We start by recalling a famous theorem of Hilbert \cite{Hilbert} that will be applied in the proof of this condition. 

\begin{theorem}[Nullstellensatz] 
Let $J$ be an ideal in $\mathbb{C}[x_{1},\dots ,x_{n}]$, and $S\subset \mathbb{C}^{n}$. Denote by $\mathcal{V}(J)$ the set of all common zeros of polynomials in $J$, and by $\mathcal{I}(S)$ the set of all polynomials in $\mathbb{C}[x_{1},\dots ,x_{n}]$ that vanish at all elements of $S$. Then 
\begin{equation*}
\mathcal{I}(\mathcal{V}(J)) = \sqrt{J} = \{f\in \mathbb{C}[x_{1},\dots,x_{n}] : f^{n} \in J \text{ for some } n \geq 1\}.
\end{equation*}
\end{theorem}

\noindent The following statement is the main theorem of this section.

\begin{theorem} \label{D}
Let $V\subset \mathbb{Z}^{n}$ be a tile. Then there is a tiling of $\mathbb{Z}^{n}$ by translates of $V$ only if there exist $(x_{1},\dots,x_{n}) \in (\mathbb{C} \setminus \{0\})^{n}$ such that $Q_{V}(x_{1}^{a}, \dots, x_{n}^{a})=0$ simultaneously for all $a$ relatively prime to $\lvert V \rvert$. 
\end{theorem}

\begin{proof}
To prove the theorem we show that if there is no $(x_{1},\dots ,x_{n}) \in (\mathbb{C}\setminus \{0\})^{n}$ such that $Q_{V}(x_{1}^{a}, \dots, x_{n}^{a}) = 0$ simultaneously for all $a$ relatively prime to $\lvert V \rvert$ then there is no tiling of $\mathbb{Z}^{n}$ by translates of $V$. \bigskip 

We start with an auxiliary statement:

\begin{quote} 
$(\ast)$ Let $\{f_{i}\}_{i \in I} \subset \mathbb{C} [x_{1}^{\pm 1},\dots,x_{n}^{\pm 1}]$ be a set of Laurent polynomials such that there exists no $(x_{1},\dots ,x_{n})\in (\mathbb{C}\setminus \{0\})^{n}$ with $f_{i}(x_{1},\dots ,x_{n})=0$ simultaneously for $i\in I$. Then there exist Laurent polynomials $p_{1},\dots ,p_{k}$ and indices $i_{1},\dots ,i_{k} \in I $ such that 
\begin{equation*}
f_{i_{1}}p_{1} + \dots + f_{i_{k}}p_{k}=1.
\end{equation*}
\end{quote}
Indeed, for each $i \in I$, consider a sufficiently large integer $n_{i}$ which makes $(x_{1}\cdots x_{n})^{n_{i}-1}f_{i}\in \mathbb{C}[x_{1},\dots,x_{n}]$; if $f_{i} \in \mathbb{C}[x_{1},\dots, x_{n}]$, then we simply set $g_{i}=(x_{1}\cdots x_{n})^{1} f_{i}$. Then $g_{i}=(x_{1}\cdots x_{n})^{n_{i}} f_{i}$ is not only a polynomial, but also a multiple of $x_{1}\cdots x_{n}$. Consider the ideal $J\subset \mathbb{C}[x_{1},\dots ,x_{n}]$ generated by the polynomials $g_{i}$. By the condition, there is no $x\in (\mathbb{C}\setminus \{0\})^{n}$ that makes $g_{i}(x)=0$ for all $i\in I$. On the other hand, $g_{i}(x)=0$ if any one of $x_{1},\dots ,x_{n}$ is zero since the polynomial is a multiple of $x_{1}\cdots x_{n}$. Thus it follows that 
\begin{equation*}
\mathcal{V}(J) = \{(x_{1},\dots ,x_{n})\in \mathbb{C}^{n} : x_{1} x_{2} \cdots x_{n}=0\}
\end{equation*}
and therefore, by Hilbert's Nullstellensatz, $x_{1}\cdots x_{n}\in \mathcal{I}(\mathcal{V}(J))=\sqrt{J}$; i.e., there exists a positive integer $m$ for which $(x_{1}\cdots x_{n})^{m}\in J$.\bigskip

Let $q_{1},\dots ,q_{k}$ and $i_{1},\dots ,i_{k}$ be the polynomials and indices which make 
\begin{align*}
(x_{1}\cdots x_{n})^{m} & = g_{i_{1}}q_{1}+\dots +g_{i_{k}}q_{k} \\
& = (x_{1}\cdots x_{n})^{n_{i_{1}}}f_{i_{1}}q_{1} + \dots + (x_{1} \cdots x_{n})^{n_{i_{k}}}f_{i_{k}}q_{k}. 
\end{align*}
Then dividing both sides by $(x_{1}\cdots x_{n})^{m}$, we get 
\begin{equation*}
1=f_{i_{1}}\frac{q_{1}}{(x_{1}\cdots x_{n})^{m-{n_{i_{1}}}}}+\dots +f_{i_{k}} \frac{q_{k}}{(x_{1} \cdots x_{n})^{m-{n_{i_{k}}}}}.
\end{equation*}
The proof of $(\ast)$ is complete. \bigskip

We are ready to prove the theorem. Assume that there is no $(x_{1},\dots ,x_{n})\in (\mathbb{C}\setminus \{0\})^{n}$ such that $Q_{V}(x_{1}^{a},\dots ,x_{n}^{a})=0$ simultaneously for all $a$ relatively prime to $\lvert V \rvert$. By $(\ast )$, we obtain
Laurent polynomials $P_{1},\dots ,P_{t}$ and integers $a_{1},\dots ,a_{t}$
relatively prime with $\lvert V \rvert$ for which 
\begin{equation}
P_{1}Q(x_{1}^{a_{1}},\dots ,x_{n}^{a_{1}})+\dots +P_{t}Q(x_{1}^{a_{t}},\dots,x_{n}^{a_{t}})=1.  \label{aa}
\end{equation}
Replacing all $x_{1},\dots ,x_{n}$ with $1$, we get 
\begin{equation}
P_{1}(1,\dots ,1)+\dots +P_{t}(1,\dots ,1) = 1 / \lvert V \rvert.  \label{a}
\end{equation}
Suppose that there exists a tiling of $\mathbb{Z}^{n}$ be translates of $V$. By (\ref{aa}), we have, for any monomial $M$, 
\begin{align*}
T(M)& = T(M(P_{1}Q(x_{1}^{a_{1}},\dots ,x_{n}^{a_{1}}) + \dots + P_{t}Q(x_{1}^{a_{t}},\dots ,x_{n}^{a_{t}}))) \\
& = T(MP_{1}Q(x_{1}^{a_{1}},\dots ,x_{n}^{a_{1}})) + \dots + T(MP_{t}Q(x_{1}^{a_{t}},\dots ,x_{n}^{a_{t}})) \\
& = P_{1}(1,\dots ,1)+\dots +P_{t}(1,\dots ,1)=1/\lvert V \rvert, \quad \text{(by Theorem~\ref{C})} 
\end{align*}
with respect to (\ref{a}). Because this differs from $0$ and $1$, we arrive at a
contradiction.
\end{proof}

\begin{remark}
To demonstrate that the above condition is only a necessary one, consider a tile $V$ given in Fig.2. We have $Q_{V}(x,y)=1+x+y+x^{2}y$, and $x=1,y=-1$ is a common root of $Q_{V}(x,y)$ and of $Q_{V}(x^{3},y^{3})$. That is, there is a non-zero common root of $Q_{V}(x^{a},y^{a})$ for each $a$ relatively prime to $4$, although there is no tiling of $\mathbb{Z}^{2}$ by $V$. However, we will prove in the next section that this  condition is a necessary and sufficient condition for tiles of a prime size.
\end{remark}

\begin{figure}[h]
\centering
\begin{tikzpicture}[scale=1]
\draw (0,0) -- (1,0) -- (1,1) -- (2,1) -- (2,0) -- (3,0) -- (3,2) -- (1,2) -- (1,1) -- (0,1) -- cycle;
\draw (2,2) -- (2,1) -- (3,1);
\end{tikzpicture} \\
\vspace{1em}
Fig.2.
\end{figure}

\noindent One of the main strength of the above theorem is that it is not limited by a special size or by a special shape of the tile. On the other hand, it is very difficult to see whether the system has a common root if the size of the tile is composite. Therefore, it will require additional research to enable one to apply this theorem toward the Golomb-Welch conjuncture. On the other hand, this theorem enables us to prove, see the next section, that there is a tiling of $\mathbb{Z}^{n}$ by translates of a prime size tile $V$ if and only if there is a lattice tiling by $V$. 

\section{Tiles of a Prime Size}

Using Polynomial Method, we show that if $V$ is a tile of a prime size then each tiling of $\mathbb{Z}^{n}$ by translates of $V$ is periodic, and that the existence of a tiling of $\mathbb{Z}^{n}$ by $V$ guarantees the existence of a lattice tiling.

\begin{theorem} \label{period}
Let $V\subset \mathbb{Z}^{n}$ be a tile, and $\mathcal{T}$ be a tiling of $\mathbb{Z}^{n}$ by translates of $V$. If $\lvert V \rvert=q$ is prime, then $q(\mathbf{v}-\mathbf{w})$ is a period of $\mathcal{T}$ for any $\mathbf{v}, \mathbf{w} \in V$. 
\end{theorem}
%

\begin{remark}
As mentioned in the introduction, Szegedy \cite{Szegedy} proved the statement by using a new technique based on loops.  Another  proof of the above statement, using similar ideas, can be found in \cite{KS}. 
\end{remark}

\begin{proof}

Consider any monomial $M$. We have
\begin{align*}
T(MQ_{V}(x_{1}^{q},\dots ,x_{n}^{q})) & \equiv T(MQ_{V}^{q})=T(MQ_{V}^{q-1}Q_{V})\\
&=(Q_{V}(1, \dots ,1))^{q-1}= q^{q-1} \equiv 0 \pmod{q}
\end{align*}
since $T(RQ_{V})=R(1, \dots, 1)$ for any polynomial $R$. On the other hand, by definition
\begin{equation*}
T(MQ_V(x_1^q, \dots, x_n^q)) = \sum_{(a_1, \dots, a_n) \in (V)} T(M x_1^{- q a_1} \cdots x_n^{- q a_n}). 
\end{equation*}
Since the sum of $\lvert V \rvert = q$ terms, each of which are either $0$ or $1$, is a multiple of $q$, we conclude that every term must be either simultaneously $0$ or simultaneously $1$. Hence for any $\mathbf{v} = (v_1, \dots, v_n)$ and $\mathbf{w} = (w_1, \dots, w_n)$ in $V$,
\begin{equation*}
T(M x_1^{- q v_1} \cdots x_n^{- q v_n}) = T(M x_1^{-q w_1} \cdots x_n^{-q w_n}) = 0 \text{ or } 1.
\end{equation*}
It follows that for any $\mathbf{x} \in \mathbb{Z}^n$, the point $\mathbf{x}$ is in $\mathcal{L}$ if and only if $\mathbf{x} + q(\mathbf{v} - \mathbf{w})$ is in $\mathcal{L}$. Therefore $q(\mathbf{v} - \mathbf{w})$ is a period of $\mathcal{T}$. \bigskip
\end{proof}

To prove a main result of this section we first state a necessary and sufficient condition, in terms of a homomorphism, for the existence of a \emph{lattice} tiling of $\mathbb{Z}^{n}$ by translates of $V$. We will use this condition in the proof of the following theorem.

\begin{theorem}[\cite{HB}] \label{Z}
Let $V$ be a subset of $\mathbb{Z}^{n}$. Then there is a lattice tiling $\mathcal{T}$ of $\mathbb{Z}^{n}$ by $V$ if and only if there is an Abelian group $G$ of order $\lvert V\rvert $ and a homomorphism $\phi : \mathbb{Z}^{n}\rightarrow G$ so that the restriction of $\phi $ to $V$ is a bijection. 
\end{theorem}

\noindent Now we are ready to show that the existence of a tiling guarantees
the existence of a lattice one. We point out that the same statement in the language of 
Abelian groups, is given, with only a hint on the proof, in \cite{Szegedy}.

\begin{theorem} \label{T1}
Let $V=\{0,v_{1}, \dots ,v_{q-1}\} \subset \mathbb{Z}^{n}$ be a prime size tile, and suppose that $\mathbf{\{v_{1}, \dots ,v_{q-1}\}}$ generate $\mathbb{Z}^{n}$. Then there exists a tiling of $\mathbb{Z}^{n}$ by translates of $V$ if and only if there is a lattice tiling of $\mathbb{Z}^{n}$ by translates of $V$. 
\end{theorem}

\begin{proof}
We assume that there exists a tiling and prove that there exists a lattice tiling. From Theorem~\ref{D}, we see that there exists a common nonzero solution to $Q_{V}(x_{1}^{a}, \dots, x_{n}^{a}) = 0$, where $a$ ranges over all integers not divisible by $q$. Let the terms of $Q_{V}$ be the monomials $m_{1},\dots ,m_{q}$, where $m_{1}=1$. If $(x_{1},\dots ,x_{n})\in \mathbb{C}^{n}$ is a common root, then for the corresponding values of $m_{1},\dots ,m_{q} \in \mathbb{C}$, we have 
\begin{equation*}
m_{1}^{a}+\dots +m_{q}^{a}=0
\end{equation*}
for all $a=1,2,\dots ,q-1$. \bigskip

Because it can be inductively deduced that the elementary symmetric polynomials 
\begin{equation*}
\sum_{i_{1}<\dots <i_{t}}m_{i_{1}}\cdots m_{i_{t}}=0
\end{equation*}
for $1\leq t<q$, we get that $m_{1},\dots ,m_{q}\in \mathbb{C}$ are roots of a polynomial which is of the form 
\[ (X - m_1) \cdots (X - m_q) = \sum_{t=0}^q (-1)^t X^{q-t} \sum_{i_{1}<\dots <i_{t}}m_{i_{1}}\cdots m_{i_{t}} = X^q - c. \]
Because $m_{1} = 1$, the constant is $c=1$, and thus $m_{1},\dots ,m_{q}$ is a permutation of $1,\zeta ,\dots ,\zeta ^{q-1}$ where $\zeta =e^{2\pi i/q}$.\bigskip

Note that since $\mathbf{\{v_{1}, \dots ,v_{q-1}\}}$ generate $\mathbb{Z}^{n}$, each of $x_{1},\dots ,x_{n}$ can be represented as a product of powers of $m_{1},\dots ,m_{q}$. This implies that the values $x_{1},\dots,x_{n}\in \mathbb{C}$ are also powers of $\zeta $. Let $x_{i}=\zeta ^{a_{i}}$ for each $i$. From the fact that the values of $m_{1},\dots ,m_{q}\in \mathbb{C}$ is a  permutation of $1,\zeta ,\dots ,\zeta ^{q-1}$, it follows that the restriction of the homomorphism $\phi :\mathbb{Z}^{n}\rightarrow \mathbb{Z} / q \mathbb{Z}$ defined by 
\begin{equation*}
(k_{1},\dots ,k_{n}) \mapsto a_{1}k_{1}+\dots +a_{n}k_{n}
\end{equation*}
restricted to $V$ is a bijection. Applying Theorem~\ref{Z} finishes the proof.
\end{proof}

\section{A Conjecture on Lattice Tilings}

It was proved in the previous section that the existence of a tiling of $\mathbb{Z}^{n}$ by a prime size tile $V$ guarantees the existence of a lattice tiling of $\mathbb{Z}^{n}$. In this section we focus on with Conjecture~\ref{Lattice} which claims that a much stronger statement is true.

\begin{conjecture} \label{Conj}
If $V = \{0,v_{1}, \dots ,v_{q-1}\} \subset \mathbb{Z}^{n}$ is of a prime size $q$ and $\{\mathbf{v}_{1}, \dots ,\mathbf{v}_{q-1}\}$ generate $\mathbb{Z}^{n}$ then each tiling of $\mathbb{Z}^{n}$ by $V$ is a lattice tiling.
\end{conjecture}

\noindent The following example exhibits that the condition: ``$\{\mathbf{v}_{1}, \dots ,\mathbf{v}_{q-1}\}$ generates $\mathbb{Z}^{n}$" cannot be replaced by a weaker assumption that $V$ is an $n$-dimensional tile.

\begin{example}
If $V=\{\mathbf{0},\mathbf{e}_{1},\dots ,\mathbf{e}_{q-2},2\mathbf{e}_{q-1}\}\subset \mathbb{Z}^{q-1}$, then the tiling 
\begin{align*}
\mathcal{T}=\{\mathbf{x} : \; &2\mid x_{q-1}\text{ and }q\mid \mathbf{x}\cdot
(1,2,\dots ,q-1); \\
\text{ or } & 2\nmid x_{q-1}\text{ and }q \mid \mathbf{x}\cdot (q-1,\dots
,2,1)\}
\end{align*}
is not lattice. 
\end{example}

First we show a rather surprising results that to prove this conjecture one can confine himself/herself to a specific tile. Later, to provide a supporting evidence, we show that the conjecture is true for all primes $q \leq 7$.

\begin{theorem} \label{Semicross}
Let $q$ be a prime. If each tiling of $\mathbb{Z}^{q-1}$ by the semi-cross $V_{q-1}=\{\mathbf{0},\mathbf{e}_{1},\dots ,\mathbf{e}_{q-1}\} $ is lattice, then each tiling of $\mathbb{Z}^{n}$ by a tile $V=\{\mathbf{0}, \mathbf{v}_{1}, \dots ,\mathbf{v}_{q-1}\}$, \ where $\{\mathbf{v}_{1}, \dots ,\mathbf{v}_{q-1}\}$ generate $\mathbb{Z}^{n}$, is a lattice tiling as well.
\end{theorem}

\begin{proof}
Let $\mathcal{T} = \{V+ l, l\in \mathcal{L}\}$ be a tiling of $\mathbb{Z}^{n}$ by a tile $V=\{\mathbf{v}_{0},\mathbf{v}_{1},\dots ,\mathbf{v}_{q-1}\}\subset \mathbb{Z}^{n}$ of a prime size $q$ such that $\{\mathbf{v}_{1}, \dots ,\mathbf{v}_{q-1}\}$ generate $\mathbb{Z}^{n}$. We show that $\mathcal{T}$ induces a tiling $\mathcal{T}_{0}$ of $\mathbb{Z}^{q-1}$ by the semi-cross $V_{q-1}$. \bigskip

Let $\phi :\mathbb{Z}^{q-1}\rightarrow \mathbb{Z}^{n}$ be a homomorphism defined by 
\begin{equation*}
(x_{1},\dots ,x_{q-1}) \mapsto \sum_{i=1}^{q-1} x_{i} \mathbf{v}_{i}.
\end{equation*}
Because of the condition that $\{\mathbf{v}_{1}, \dots ,\mathbf{v}_{q-1}\}$ generate $\mathbb{Z}^{n}$, the homomorphism $\phi $ is surjective. \bigskip

Let 
\begin{equation*}
\mathcal{T}_{0}=\phi ^{-1}(\mathcal{T})=\bigg\{ (x_{1},\dots ,x_{q-1}) \in \mathbb{Z}^{q-1}:\sum_{i=1}^{q-1}x_{i}\mathbf{v}_{i}\in \mathcal{T} \bigg\} .
\end{equation*}
Since exactly one of $\mathbf{x},\mathbf{x}+\mathbf{v}_{1},\dots ,\mathbf{x}+\mathbf{v}_{q-1}$ is contained in $\mathcal{T}$ for each $\mathbf{x}\in \mathbb{Z}^{n}$, exactly one of $\mathbf{x},\mathbf{x}+\mathbf{e}_{1},\dots ,\mathbf{x}+\mathbf{e}_{q-1}$ is contained in $\mathcal{T}_{0}$ for each $\mathbf{x}\in \mathbb{Z}^{q-1}$. Thus $\mathcal{T}_{0}$ is actually a tiling of $\mathbb{Z}^{q-1}$ by $V_{0}=\{\mathbf{0},\mathbf{e}_{1},\dots ,\mathbf{e}_{q-1}\}$.\bigskip

Because $\phi$ is surjective, the image is $\phi (\mathcal{T}_{0})=\mathcal{T}$. If $\mathcal{T}_{0}$ is a subgroup of $\mathbb{Z}^{q-1}$, then $\mathcal{T}$ also becomes a subgroup of $\mathbb{Z}^{n}$. Thus, it is sufficient to show that $\mathcal{T}_{0}$ is always a lattice
tiling. \bigskip 
\end{proof}
\bigskip

The following conjecture is equivalent to Conjecture~\ref{Conj}.

\begin{conjecture}
Let $V=\{0,\mathbf{v}_{1},\dots ,\mathbf{v}_{q-1}\}\subset \mathbb{Z}^{n}$ of a prime size $q$ tiles $\mathbb{Z}^{n}$ by translates, and $\{\mathbf{v}_{1},\dots ,\mathbf{v}_{q-1}\}$ generate $\mathbb{Z}^{n}$. Then there is a unique tiling, up to a congruency, of $\mathbb{Z}^{n}$ by $V$ and this tiling is lattice.
\end{conjecture}

Indeed, if there were two non-congruent lattice tilings of $%
\mathbb{Z}
^{n\text{ }}$ by $V,$ then the induced tilings of $%
\mathbb{Z}
^{q-1\text{ }}$by semi-crosses would be non-congruent as well. However, by
Theorem 18, all lattice tilings of $%
\mathbb{Z}
^{q-1\text{ }}$ by semi-cross are congruent.
\bigskip

To provide supporting evidence we show that:

\begin{theorem} \label{Con}
Let $V = \{0,v_{1}, \dots ,v_{q-1}\}$ be a tile of a prime size $q \leq 7$ such that $\{\mathbf{v}_{1}, \dots ,\mathbf{v}_{q-1}\}$ generate $\mathbb{Z}^{n}$. Then each tiling of $\mathbb{Z}^{n}$ by $V$ is lattice.
\end{theorem}

\noindent To facilitate our discussion we introduce new notions and notation, and state several auxiliary results. Let $\mathcal{T}=\{V_{q-1}+l;l\in \mathcal{L}\}$ be a tiling of $\mathbb{Z}^{q-1}$ by semi-crosses. We use terminology of coding theory; that is, the elements of $\mathbb{Z}^{q-1}$ will be called words and the elements of $\mathcal{L}$, the centers of semi-crosses in $\mathcal{T}$, will be called codewords. \bigskip 

By a word of type $[m_{1}^{\alpha _{1}}, \dots ,m_{s}^{\alpha _{s}}]$ we mean a word having $\alpha _{1}$ coordinates equal to $m_{1}$, \dots, $\alpha _{s}$ coordinates equal to $m_{s}$, the other coordinates equal to zero. Let $W, Z$ be words, and the word $Z-W$ is of type $[m_{1}^{\alpha_{1}}, \dots ,m_{s}^{\alpha _{s}}]$. Then $Z$ will be called a word of type $[m_{1}^{\alpha_{1}}, \dots ,m_{s}^{\alpha _{s}}]$ with respect to $W$. For $W=O$, we simplify the language and call $Z$ shortly a word of type $[m_{1}^{\alpha _{1}}, \dots ,m_{s}^{\alpha _{s}}]$. Further, we will say that a word $V$ is covered by a codeword $W$ if $V$ belongs to the semi-cross centered at $W$. Finally, two words $A,B$ coincide in $t$ coordinates, if they have the same value in $t$ \underline{\textit{non-zero}} coordinates.\bigskip 

The following theorem constitutes a crucial tool for proving the main result of this section.

\begin{theorem} \label{type}
Let $\mathcal{T}$ be a tiling of $\mathbb{Z}^{p-1}$ by semi-crosses. Then, for a prime $p$ and any $k<p$, we have 
\begin{equation*}
T\bigg(\sum_{i_{1}<\dots <i_{k}}x_{i_{1}}x_{i_{2}}\cdots x_{i_{k}}\bigg)=\frac{\binom{p-1}{k}-(-1)^{k}}{p}+(-1)^{k}T(1).
\end{equation*}
In other words, if $O$ is a codeword then there are $\frac{1}{p}(\binom{p-1}{k}+(p-1)(-1)^{k})$ codewords of type $[1^{k}]$, otherwise there are $\frac{1}{p}(\binom{p-1}{k}-(-1)^{k})$ codewords of type $[1^{k}]$.
\end{theorem}

\begin{proof}
For convenience, we let 
\begin{equation*}
e_{k}=\sum_{i_{1}<\dots <i_{k}}x_{i_{1}}x_{i_{2}}\cdots x_{i_{k}}
\end{equation*}
denote the elementary symmetric polynomials. We use induction on $k$. 

For $k=1$ we get 
\begin{align*}
T(e_{1}) &= T(x_{1}+ \dots +x_{q-1}+1-1) = T(1+x_{1}+ \dots +x_{q-1})-T(1) \\
& =1-T(1) = \frac{\binom{p-1}{1}-(-1)^{1}}{p}+(-1)^{1}T(1). 
\end{align*}

Suppose now that the identity $T(e_{j})=\frac{1}{p}(\binom{p-1}{j}-(-1)^{j})+(-1)^{j}T(1)$ is true for all $1\leq j < k$. Consider the identity 
\begin{equation*}
\big(\sum x_{i}^{k}\big)-e_{1}\big(\sum x_{i}^{k-1}\big) + \dots + (-1)^{k-1}e_{k-1}\big(\sum x_{i}\big)+(-1)^{k}ke_{k}=0.
\end{equation*}
Note that it is true since all terms of the form $x_{i_{1}} \cdots x_{i_{j}}x_{i_{j+1}}^{k-j}$ are added and subtracted exactly once. It follows from this identity that 
\begin{align*}
0 & =T\big(\sum x_{i}^{k}\big)-T\big(e_{1}\big(\sum x_{i}^{k-1}\big)\big) +\dots +(-1)^{k}kT(e_{k}) \\
& =\sum_{j=0}^{k-1}(-1)^{j}T\Big(e_{j}\Big(\sum x_{i}^{k-j}\Big)\Big)+(-1)^{k}kT(e_{k}) \\
& =\sum_{j=0}^{k-1}(-1)^{j}\bigg[\binom{p-1}{j}-T(e_{j})\bigg]+(-1)^{k}kT(e_{k}) \\
& =\sum_{j=0}^{k-1}\bigg[(-1)^{j}\bigg(\binom{p-1}{j}-\frac{1}{p}\binom{p-1}{j}\bigg)+\frac{1}{p}-T(1)\bigg]+(-1)^{k}kT(e_{k}) \\
& =\frac{p-1}{p}\sum_{j=0}^{k-1}(-1)^{j}\binom{p-1}{j}+k(\tfrac{1}{p}-T(1))+(-1)^{k}kT(e_{k}) \\
& =\frac{p-1}{p}\sum_{j=0}^{k-1}(-1)^{j}\Big(\binom{p-2}{j}+\binom{p-2}{j-1}\Big)+k\big(\tfrac{1}{p}-T(1)+(-1)^{k}T(e_{k})\big) \\
& =(-1)^{k-1}\frac{p-1}{p}\binom{p-2}{k-1}+k\big(\tfrac{1}{p}-T(1)+(-1)^{k}T(e_{k})\big)
\end{align*}
since 
\begin{align*}
T\big(e_{j}\big(\sum x_{i}^{k-j}\big)\big)& =T\big(e_{j}\big(1+\sum
x_{i}^{k-j}\big)\big)-T(e_{j}) \\
& =e_{j}(1,\dots ,1)-T(e_{j})=\binom{p-1}{j}-T(e_{j}).
\end{align*}
Hence we get 
\begin{align*}
T(e_{k})&=(-1)^{k+1}(\tfrac{1}{p}-T(1))+\frac{p-1}{kp}\binom{p-2}{k-1}\\
&=\frac{1}{p}\Big(\textstyle\binom{p-1}{k}-(-1)^{k}\Big)+(-1)^{k}T(1). \qedhere
\end{align*}
\end{proof}

\noindent It is possible to prove a much more general statement than the
above theorem; we skip the proof here as it is quite long and involved, and
we do not need the statement in what follows.

\begin{theorem}
Let $\mathcal{T}$ be a tiling of $\mathbb{Z}^{p-1}$ by semi-crosses, where $p$ is a prime. Then, for any $m_{1}, \dots ,m_{t}$ and $\alpha _{1}, \dots ,\alpha_{t}$, there are constants $C$ and $C^{\prime }$ depending only on $m_{i}$'s and $\alpha_{i}$'s such that the number of codewords of type $[m_{1}^{\alpha _{1}}, \dots ,m_{t}^{\alpha _{t}}]$ is $C$ if $O$ is a codeword, otherwise it is $C^{\prime}$.
\end{theorem}

\noindent Now we state several corollaries of Theorem~\ref{type}.

\begin{corollary}
Let $W$ be a codeword. Then there are $\frac{1}{p}(\binom{p-1}{k}+(p-1)(-1)^{k})$ codewords of type $[1^{k}]$, and also of type $[-1^{k}]$ with respect to $W$.
\end{corollary}

\begin{proof}
It suffices to consider the ``shifted" tiling $\mathcal{T}_{W}=\{V_{q-1}+l,l \in \mathcal{L}-W\}$, and $\{V_{q-1}+l,l\in -\mathcal{L\}}$, a reflection of tiling $\mathcal{T}\}$.
\end{proof}

\begin{corollary} \label{C0}
For any $m,k \leq q-1$, the number of codewords of type $[m^{k}]$ with respect to $W$, equals the number of codewords of type $[1^{k}]$ with respect to $W$.
\end{corollary}

\begin{proof}
It is sufficient to prove the statement only for $W=O.$ To determine the number of codewords of type $[m^{k}]$ we need to
calculate $T(e_{k}^{(m)}),$ where $e_{k}^{(m)}=\sum_{i_{1}<\dots <i_{k}}x_{i_{1}^m}x_{i_{2}^m}\cdots x_{i_{k}^m}$. Thus, $%
e_{k}^{(1)}=e_{k}$ as defined above. It is not difficult to see that after
substituting $y_{i}=x_{i}^{m}$ one can apply verbatim the proof of Theorem \ref{type}. 
\end{proof}

\noindent Substituting $k=q-1$, we get:

\begin{corollary}
\label{C1}For each word $W$, $W$ is a codeword if and only if $W \pm (1,1, \dots ,1)$ is a codeword as well.
\end{corollary}

\noindent Applying Corollary \ref{CC} we have

\begin{corollary}
\label{C2}Let $W$ be a codeword. Then, for any $a\leq q-1$, there is no
codeword of type $[a^{1}]$ and of type $[a^{1},-a^{1}]$ with respect to $W$.
\end{corollary}

\noindent Finally,

\begin{corollary} \label{C3}
Let $W$ be a codeword. Then there are $n=\frac{q-1}{2}$ codewords $U_{1}, \dots ,U_{n}$ \ of type $[1^{2}]$, and $n=\frac{q-1}{2}$ codewords $U_{1}^{\prime}, \dots ,U_{n}^{\prime}$ of type $[-1^{2}]$ with respect to $W$. In addition, 
\begin{equation*}
\sum_{i=1}^{n} U_{i}-W=I, \quad \sum_{i=1}^{n}U_{i}^{\prime}-W=-I,
\end{equation*}
where $I = (1, 1, \dots, 1)$. 
\end{corollary}

\begin{proof}
By Theorem \ref{type}, $n=\frac{q-1}{2}$. To see the other part of the statement it suffices to note that if two words $U_{i}-W$ and $U_{i}-W$ coincided in a coordinate then the semi-crosses centered at $U_{i}$ and $U_{j}$ would not be disjoint.
\end{proof}

\bigskip

\begin{proof}[Proof of Theorem~\ref{Con}]
By Theorem~\ref{Semicross} it is sufficient to prove that each tiling of $\mathbb{Z}^{q-1}$ by translates of $V=\{0,e_{1}, \dots ,e_{q-1}\}$ is lattice. We start by introducing additional notation and notions. We denote by $I$ the word $(1,1, \dots ,1)$. For a word $W=(a_{1},a_{2},a_{3}, \dots ,a_{q-1})$, by $\pi (W)$ we mean the word obtained by the shift of coordinates of $W$, i.e., $\pi(W) = (a_{2},a_{3}, \dots ,a_{q-1},a_{1})$; further we put $\langle W \rangle := \{W, \pi (W),\pi^{2}(W), \dots ,\pi ^{q-2}(W)\}$, the set of all shifts of $W$. \bigskip 

\noindent $q=2$. Trivially, each tiling of $\mathbb{Z}^{1}$ by the tile $V_{1} = \{0,e_{1}\}$ is lattice. \bigskip 

\noindent $q=3$. By Corollary \ref{C1}, for all integers $n$, $nI$ is a codeword, and from periodicity of $\mathcal{L}$, if $W$ is a codeword then $W + 3n e_{1}$ is a codeword as well. Thus, $\mathcal{L}$ contains a lattice $F$ generated by $I$ and $3e_{1}$. However, $F=\mathcal{L}$ as 
\begin{equation*}
\det \begin{vmatrix} 1 & 1 \\ 3 & 0 \end{vmatrix} =-3.
\end{equation*}

In what follows, there will be several statements formulated for a general codeword $W$ but we will prove them all without loss of generality only for $W=O$. Further, we point out, that with respect to a cyclic property, it suffices to prove statements given below only for one codeword from a set $\langle V \rangle$.

\bigskip

\noindent $q=5$. Let $W$ be a codeword. By Corollary~\ref{C3}, there are two codewords $A, B$ of type $[1^{2}]$, $A+B=I$, and two codewords $C,D$ of type $[-1^{2}]$, $C+D=-I$, with respect to $W$; we denote them by $\mathcal{U}_{2}^{+}(W)$ and $\mathcal{U}_{2}^{-}(W)$, respectively. To simplify the proof, we assume without loss of generality that $\mathcal{U}_{2}^{+}(O)=\langle (1,0,1,0) \rangle$. There are $6$ words of type $[1^{2}]$, each of them covered by a codeword either of type $[1^{2}]$ or of type $[1^{2},-1]$. Hence, as there are $2$ codewords of type $[1^{2}]$, there have to be $4$ codewords of type $[1^{2},-1]$, we denote them by $\mathcal{U}_{3}^{+}(W)$; the same is true for codewords of type $[-1^{2},1]$, they will be denoted $\mathcal{U}_{3}^{-}(W)$. The following straightforward claim will be applied repeatedly in the proof.\bigskip 

\noindent \textbf{Claim 1.} Each codeword in $\mathcal{U}_{2}^{+}(W)$ determines uniquely the other codeword in $\mathcal{U}_{2}^{+}(W)$. In addition, codewords in $\mathcal{U}_{2}^{+}(W)$, and one codeword in $\mathcal{U}_{3}^{+}(W)$, determine uniquely the other codewords in $\mathcal{U}_{3}^{+}(W)$. The same is true for the ``$-$'' part. \bigskip 

\noindent For example, if $\mathcal{U}_{2}^{+}(O) = \langle (1,0,1,0) \rangle$, then we have 
$\mathcal{U}_{3}^{+}(O)=\langle (1,1,-1,0) \rangle$ or $\mathcal{U}_{3}^{+}(O)=\langle (1,1,0,-1) \rangle$. Thus, one codeword in $\mathcal{U}_{3}^{+}(O)$ determines in a unique way the set $\mathcal{U}_{3}^{+}(O)$. \bigskip 

\noindent \textbf{Claim 2.} Let $W$ be a codeword. Then, for each $A\in \mathcal{U}_{2}^{+}(W)$, 
\begin{equation*}
\mathcal{U}_{2}^{+}(W+A)=\mathcal{U}_{2}^{+}(W).
\end{equation*}

\begin{proof}[Proof of Claim 2] 
Let $W=O$ without loss of generality. As $I$ is a codeword, it means that $(0,1,0,1)\in \mathcal{U}_{2}^{+}((1,0,1,0))$, and therefore $\mathcal{U}_{2}^{+}((1,0,1,0)) = \mathcal{U}_{2}^{+}(O)$. Further, since $\langle (1,0,1,0) \rangle -I = -\langle (1,0,1,0) \rangle$ are codewords, we get
\begin{equation*}
\mathcal{U}_{2}^{-}(O) = -\mathcal{U}_{2}^{+}(O),
\end{equation*}
which implies 
\begin{equation*}
\mathcal{U}_{2}^{-}(W+A) = \mathcal{U}_{2}^{-}(W), \text{ for any codeword }A \in \mathcal{U}_{2}^{-}(W).
\end{equation*}
It follows that $\mathcal{U}_2^+ (W+A) = \mathcal{U}_2^- (W)$. 
\end{proof}

\noindent Then, by a straightforward induction, tiling $\mathcal{T}$ contains a lattice generated by $\langle (1,0,1,0) \rangle$. Now we will show that $\mathcal{T}$ contains a lattice generated by $\langle (1,0,1,0) \rangle$ and any codeword in $\mathcal{U}_{3}^{+}(O)$. \bigskip 

\noindent \textbf{Claim 3.} $\mathcal{U}_{3}^{-}(O)=-\mathcal{U}_{3}^{+}(O)$.

\begin{proof}[Proof of Claim 3] 
Assume by contradiction that $\mathcal{U}_{3}^{-}(O)\neq - \mathcal{U}_{3}^{+}(O)$, say without loss of generality, $\mathcal{U}_{3}^{+}(O) = \langle (1,1,-1,0) \rangle$, and $\mathcal{U}_{3}^{-}(O) = -\langle (1,1,0,-1) \rangle$. Then, $U=(1,1,-1,0) \in \mathcal{U}_{3}^{+}(O)$, and $V=(1,-1,-1,0) \in \mathcal{U}_{3}^{-} (O)$. However, this contradicts Corollary~\ref{C2} as $U-V$ is of type $[2^{1}]$. \bigskip 
\end{proof}

\noindent \textbf{Claim 4.} For any $A\in \mathcal{U}_{2}^{+}(W)$, 
\begin{equation*}
\mathcal{U}_{3}^{+}(W+A)=\mathcal{U}_{3}^{+}(W),
\end{equation*}
and, for any $A \in \mathcal{U}_{3}^{+}(W)$,
\begin{equation*}
\mathcal{U}_{2}^{+}(W+A)=\mathcal{U}_{2}^{+}(W).
\end{equation*}

\begin{proof}[Proof of Claim 4] 
Let $W=O$. Assume without loss of generality that $\mathcal{U}_{3}^{+}(O)=\langle (1,1,-1,0) \rangle$. It is sufficient to notice that $B+C-I$ is a codeword for each $B \in \langle (1,0,1,0) \rangle$, and $C \in \langle (1,1,-1,0) \rangle$. \bigskip 
\end{proof}

\noindent In view of Claim~3, Claim~4 is true also for the ``$-$'' part. \bigskip 

\noindent \textbf{Claim 5.} Let $\mathcal{U}_{2}^{+}(W)=\langle (1,0,1,0) \rangle$. Then $\mathcal{U}_{3}^{+}(W+A)=\mathcal{U}_{3}^{+}(W)$ for any $A \in \mathcal{U}_{3}^{+}(W)$. 

\begin{proof}[Proof of Claim 5]
Assume that $\mathcal{U}_{3}^{+}(A) \neq \mathcal{U}_{3}^{+}(O)$ for an $A\in \mathcal{U}_{3}^{+}(O)$, say without loss of generality that $\mathcal{U}_{3}^{+}(O)=\langle (1,1,-1,0) \rangle$, and $\mathcal{U}_{3}^{+}((1,1,-1,0))= \langle (1,1,0,-1) \rangle$. Then $(1,1,-1,0) + (0,-1,1,1) = (1,0,0,1)$ would be a codeword, a contradiction since $\mathcal{U}_2^+(W) = \langle (1,0,1,0) \rangle$. \bigskip 
\end{proof}

As above, by Claim~3, Claim 5 is true also for the ``$-$'' part. Assume without loss of generality that $\langle (1,1,-1,0) \rangle$ are codewords. Let $n,m,k$ be integers. By Claim 2, $W=n(1,0,1,0)+m(0,1,0,1)$ is a codeword, and $\mathcal{U}_{2}^{\pm} (W) = \pm \langle (1,0,1,0) \rangle$. By Claim 4 and Claim 5, and a straightforward induction, $W+k(1,1,-1,0)$ is a codeword. Finally, taking into account that $\mathcal{T}$ is periodic with $5e_{1}$, we have that $\mathcal{T}$ contains a lattice $\mathcal{R}$ generated by $\langle (1,0,1,0) \rangle$, $(1,1,-1,0)$, and $5e_{1}$. The determinant of the matrix whose rows are the given four vectors equals $\pm 5$, thus $\mathcal{R}$ contains all codewords of $\mathcal{T}$. The proof for $q=5$ is complete.

\bigskip

\noindent $q=7$. By Theorem~\ref{type}, for each codeword $W$ there are $3$ codewords of type $[1^{2}]$ and $2$ codewords of type $[1^{3}]$ with respect to $W;$ we denote these codewords by $\mathcal{U}_{2}(W)$ and $\mathcal{U}_{3}(W)$, respectively.

\begin{lemma}
Let $W$ be a codeword. Then 
\begin{equation*}
\sum_{V\in \mathcal{U}_{2}(W)} V = \sum_{Z\in \mathcal{U}_{3}(W)} Z = (1,1,1,1,1,1).
\end{equation*}
\end{lemma}

\begin{proof}[Proof of Lemma] 
Let $W=O$. The statement is obvious for $\mathcal{U}_{2}(W)$. To show it for $\mathcal{U}_{3}(W)$, we need in fact to prove that the two codewords $Z_{1}$ and $Z_{2}$ of type $[1^{3}]$ do not coincide in any coordinate. $Z_{1}$ and $Z_{2}$ cannot coincide in two coordinates, otherwise the two semi-crosses centered at $Z_{1}$ and $Z_{2}$ would have a non-empty intersection. So assume by contradiction that $Z_{1}$ and $Z_{2}$ coincide in exactly one coordinate. Let, without loss of generality, $Z_{1}=(1,0,1,0,1,0)$ and $Z_{2}=(1,0,0,1,0,1)$. Now we show that \bigskip 

\noindent \textbf{Claim 6.} For each codeword $W$ there are 
\begin{itemize}
\itemsep 0em
\item[(i)] $3$ codewords of type $[1^{4}]$; for each $1\leq i\leq 6$, exactly one
of the three codewords has the $i$-th coordinate equal to $0;$ 
\item[(ii)] $6$ codewords of type $[1^{4},-1^{1}]$ with respect to $W$.
No two of these codewords coincide in the coordinate whose value is $-1$.
\end{itemize}

\begin{proof}[Proof of Claim 6] 
Let $W=0$. 
\begin{itemize}
\itemsep 0em
\item[(i)] There is no codeword of type $[1^{5}]$. So all $6$ words of this type are covered by the codewords of type $[1^{4}]$, the statement follows. 
\item[(ii)] There are $15$ words of type $[1^{4}]$; $3$ of them are codewords, $6$ of them are covered by codewords of type $[1^{3}]$ (regardless of in how many coordinates $Z_{1}$ and $Z_{2}$ might coincide). Thus the remaining $6$ words have to be covered by codewords of type $[1^{4},-1^{1}]$. Clearly, each codeword of type $[1^{4},-1^{1}]$ covers only one word of type $[1^{4}]$, thus there are $6$ codewords of type $[1^{4},-1^{1}]$. Two semi-crosses centered at codewords of type $[1^{4},-1^{1}]$ coinciding in the coordinate whose value is $-1$ would have a non-empty intersection. \qedhere
\end{itemize}
\end{proof}

\noindent Now we are ready to finish the proof of our lemma. Let $Z$ be a codeword of type $[1^{4}]$ whose second coordinate equals $0$. The three codewords $Z$, $Z_{1}$, and $Z_{2}$ cover all $5$ words of type $[1^{4}]$ whose second coordinate equals to $0$. However, by Claim~6, there is a codeword $W$ of type $[1^{4},-1^{1}]$ whose second coordinate equals $-1$. Clearly, $W$ also covers a word of type $[1^{4}]$ whose second coordinate is $0$, a contradiction. The proof is complete. \bigskip 
\end{proof}

We note that any codeword of type $[1^{3}]$ coincides with each codeword of type $[1^{2}]$ in precisely one coordinate. We will assume without loss of generality that $\mathcal{U}_{2}(O)=\langle (1,0,0,1,0,0) \rangle$, and $\mathcal{U}_{3}(O) = \langle (1,0,1,0,1,0) \rangle$. Also we set, $B_{1}=(1,0,0,1,0,0)$, and $B_{2}=\pi(B_{1}), B_{3}=\pi ^{2}(B_{2})$.\bigskip 

\noindent \textbf{Claim 7.} Let $W$ be a codeword. Then $\mathcal{U}_{2}(W+A) = \mathcal{U}_{2}(W)$ for all $A\in \mathcal{U}_{2}(W)$. In particular, the $3$ codewords of type $[1^{4}]$ are $\langle (0,1,1,0,1,1) \rangle$, and the $3$ codewords of type $[2^{2}]$ are $2 \cdot \mathcal{U}_{2}(O)= \langle (2,0,0,2,0,0) \rangle.$ Further, tiling $\mathcal{T}$ contains a lattice generated by codewords $\langle (1,0,0,1,0,0)\rangle$. 

\begin{proof}[Proof of Claim 7] 
Let $W=O$. To prove the statement it
suffices to show that the $3$ codewords of type $[1^{4}]$ are $%
\,\langle (0,1,1,0,1,1) \rangle$.\bigskip 

By Corollary~\ref{C0}, there are $3$ codewords of type $[2^{2}]$. We show that these codewords are $2 \cdot \langle (1,0,0,1,0,0) \rangle =\langle (2,0,0,2,0,0) \rangle$. Assume that there is codeword $W$ of type $[2^{2}]$ such that $W \notin \langle (2,0,0,2,0,0) \rangle$. Let $V$ be a codeword of type $[1^{2},-1]$ such that $V$ and $\frac{1}{2} W$ coincide in two coordinates. Then $W-V=Z$ is of type $[1^{3}]$. Hence, $Z \in \mathcal{U}_{3}(V)$, and $Z^{\prime } = I-Z \in \mathcal{U}_{3}(V)$ as well. Thus, $V+Z^{\prime}$ is a codeword, and it is of type $[1^{5},-1]$. This is a contradiction as $I$ is a codeword. Therefore, $B_{i}\in \mathcal{U}_{2}(B_{i})$, $i=1,2,3$. Let $C_{2},C_{3}$ be the other two codewords in $\mathcal{U}_{2}(B_{1})$. Then $C_{2}+C_{3}=(0,1,1,0,1,1)$ and $B_{1}+C_{i}$, $i=2,3$, are codewords of type $[1^{4}]$. By Claim 6, $(0,1,1,0,1,1)$ has to be a codeword. However, this means that $B_{2} \in \mathcal{U}_{2}(B_{3})$, and therefore also $B_{1}\in \mathcal{U}_{2}(B_{3})$; that is, $\mathcal{U}_{2}(B_{3})=\mathcal{U}_{2}(O)$. By the same argument we get $\mathcal{U}_{2}(B_{i})=\mathcal{U}_{2}(O)$ for $i=2,3$. But then $(1,1,0,0,1,1)$ is a codeword, a contradiction. The proof of the first part is complete. Clearly, as $\mathcal{U}_{2}(B_{1})=\mathcal{U}_{2}(O)$, $\langle (1,1,0,1,1,0) \rangle$ are the three codewords of type $[1^{4}]$. The final part of the proof follows by a straightforward induction.\bigskip 
\end{proof}

Let $\mathcal{U}_{2}(W) = \langle (1,0,0,1,0,0) \rangle$, and $\mathcal{U}_{3}(W)=\langle (1,0,1,0,1,0) \rangle$. Then the three codewords of type $[1^{4}]$ are $\langle (1,1,0,1,1,0) \rangle$. Further, the $6$ codewords of type $[1^{4},-1^{1}]$ with respect to $W$ are either $\langle (1,1,1,1,-1,0) \rangle$, or $\langle (1,1,1,1,0,-1) \rangle$. We denote these codewords by $\mathcal{U}_{4}(W)$. \bigskip 

\noindent \textbf{Claim 8.} Let $W$ be a codeword, $\mathcal{U}_{2}(W) = \langle (1,0,0,1,0,0) \rangle$, and $\mathcal{U}_{3}(W)=\langle (1,0,1,0,1,0) \rangle$. If $\mathcal{U}_{4}(W)=\langle (1,1,1,1,-1,0) \rangle$, then the $6$ codewords of type $[1^{3},-1]$ are $\langle (1,1,1,0,0,-1) \rangle$, and the $12$ codewords of type $[1^{2},-1^{1}]$ are $\langle (1,1,0,-1,0,0) \rangle$, and $\langle (1,-1,1,0,0,0) \rangle$. If $\mathcal{U}_{4}(W) = \langle (1,1,1,1,0,-1) \rangle$, then the $6$ codewords of type $[1^{3},-1]$ are $\langle (1,1,1,-1,0,0) \rangle$, and the $12$ codewords of type $[1^{2},-1^{1}]$ are $\langle (1,1,0,0,-1,0) \rangle$, and $\langle (1,-1,1,0,0,0) \rangle$.

\begin{proof}[Proof of Claim 8]
Let $W=O$. The six words of type $[1^{4}]$ that are covered by codewords of type $[1^{4},-1]$ are $\langle (1,1,1,1,0,0) \rangle$. By Claim~6, the codewords of type $[1^{4},-1]$ are either $\langle (1,1,1,1,-1,0) \rangle$ or $\langle (1,1,1,1,0,-1) \rangle$. Assume the former case. The six words of type $[1^{3}]$ that are covered by codewords of type $[1^{3},-1]$ are $\langle (1,1,1,0,0,0) \rangle$. Thus, the codewords of type $[1^{3},-1]$ have to be $\langle (1,1,1,0,0,-1) \rangle$, otherwise the semi-crosses centered at $\langle (1,1,1,1,-1,0) \rangle$ and at codewords of type $[1^{3},-1]$ would not be disjoint. Finally, the $12$ words of type $[1^{2}]$ covered by codewords of type $[1^{2},-1]$ are $\langle (1,1,0,0,0,0) \rangle$ and $\langle (1,0,1,0,0,0) \rangle$. Because of codewords $\langle (1,1,1,0,0,-1) \rangle$, the codewords covering words $\langle (1,1,0,0,0,0) \rangle$ and $\langle (1,0,1,0,0,0) \rangle$ are the codewords $\langle (1,1,x,y,0,0) \rangle$ and $\langle (1,z,1,v,w,0) \rangle$, where one of $x,y$ is $-1$ and the other equals $0$, also one of $z,v,w$ equals $-1$, the other two are $0$'s. With respect to $(1,0,1,0,1,0)$, it is $w=0$. If $x=-1$, then semi-crosses $\langle (1,1,x,y,0,0) \rangle$ and $\langle (1,z,1,v,0,0) \rangle$ would intersect. For the same reason it is impossible to have $y=v=-1$. Using the same ideas one can show the other part of the claim for the latter case. \bigskip 
\end{proof}

\noindent \textbf{Claim 9.} If $W$ is a codeword described in Claim~8, then $-W$ is a codeword as well.

\begin{proof}[Proof of Claim 9] 
We know that if $W$ is a codeword then $W \pm I$ as also a codeword. Hence, 
\begin{equation*}
\langle (1,1,0,1,1,0) \rangle - I = -\langle (1,0,0,1,0,0) \rangle = - \mathcal{U}_{2}(W)
\end{equation*}
are codewords. Further, 
\begin{equation*}
\mathcal{U}_{3}(W)-I=-\langle (1,0,1,0,1,0) \rangle =-\mathcal{U}_{3}(W)
\end{equation*}
are codewords as well. Applying Claim~8 to $-\mathcal{U}_{2}(W)$ and $-\mathcal{U}_{3}(W)$ we need only to show that if $\langle (1,1,1,1,-1,0) \rangle$ ($\langle (1,1,1,1,0,-1) \rangle$) are codewords, then $-\langle (1,1,1,1,-1,0) \rangle$, ($-\langle (1,1,1,1,0,-1) \rangle$) are again codewords. Assume by contradiction that $\langle (1,1,1,1,-1,0) \rangle$ and $-\langle (1,1,1,1,0,-1)\rangle $ are codewords. Then, by Claim~8, $\langle (1,1,1,0,0,-1) \rangle$ and $\langle (-1,-1,-1,1,0,0) \rangle$ are codewords as well. Set $Z = (1,1,1,0,0,-1)$, and $U = (-1,-1,1,0,0,-1)$. Then $Z - U = (2,2,0,0,0,0)$ is a codeword of type $[2^{2}]$ with respect to $U$. This in turn implies, see Claim~7, that $(1,1,0,0,0,0)$ is a codeword of type $[1^{2}]$ with respect to $U$, that is, $U+(1,1,0,0,0,0)=(0,0,1,0,0,-1)$ is a codeword. However, this contradicts Corollary~29. The proof is complete. \bigskip 
\end{proof}

\noindent \textbf{Claim 10.} $\mathcal{U}_{3}(W+B)=\mathcal{U}_{3}(W)$ for all $B \in \mathcal{U}_{2}(W)$, and $\mathcal{U}_{2}(W+B)=\mathcal{U}_{2}(W)$ for all $B \in \mathcal{U}_{3}(W)$.

\begin{proof}[Proof of Claim 10] 
Let $W=O$. It suffices to note that 
\textit{for each }$U\in \langle (1,0,1,0,1,0) \rangle$, and each $Z \in \langle (1,0,0,1,0,0) \rangle$, $U+Z-I$ is a codeword.\bigskip 
\end{proof}

\noindent \textbf{Claim 11.} $\mathcal{U}_{3}(W+B)=\mathcal{U}_{3}(W)$ for all $B\in \mathcal{U}_{3}(W)$.

\begin{proof}[Proof of Claim 11] Let $W=O$. Assume by contradiction that $\mathcal{U}_{3}(B)\neq \mathcal{U}_{3}(O)$. Then there is $C$ in $\mathcal{U}_{3}(B)$ such that $B_{4}=(1,0,1,0,1,0)$ coincides with $C$ only in one coordinate. This in turn implies that $B_{4}+C-I$ is of type $[1,-1]$, a contradiction.\bigskip 
\end{proof}

\noindent By Claim 8, 10, and 11 we have that the tiling $\mathcal{T}$ contains a lattice generated by $\langle (1,0,0,1,0,0) \rangle$ and $(1,0,1,0,1,0)$. \bigskip

\noindent \textbf{Claim 12.} Let $W$ be a codeword. Then $\mathcal{U}_{2}(W+B)=\mathcal{U}_{2}(W)$, $\mathcal{U}_{3}(W+B)=\mathcal{U}_{2}(W)$, and $\mathcal{U}_{4}(W+B)=\mathcal{U}_{4}(W)$ for all $B\in \mathcal{U}_{4}(W)$.

\begin{proof}[Proof of Claim 12]
As this proof uses the same techniques as presented above we leave it for the reader. \bigskip 
\end{proof}

With Claim~12 in hands we know that $\mathcal{T}$ contains a
lattice generated by $\langle (1,0,0,1,0,0) \rangle ,(1,0,1,0,1,0)$, and a vector from $%
\mathcal{U}_{4}(W)$, say either $(1,1,1,1,-1,0)$ or $(1,1,1,1,0,-1)$. In
addition, $\mathcal{T}$ is periodic with $7e_{1}$. The determinant of a
matrix whose rows are these $7$ vectors is $\pm 7;$ that is, the lattice
contains all codewords of $\mathcal{T}$.\bigskip 
\end{proof}

As an immediate consequence of Theorem~\ref{Z} we get:

\begin{corollary}
let $V=\{0, \mathbf{v}_{1},\dots, \mathbf{v}_{q-1} \} \subset \mathbb{Z}^{n}$ of a prime size $q\leq 7$ tiles $\mathbb{Z}^{n}$ by translates, and $\{\mathbf{v}_{1},\dots ,\mathbf{v}_{q-1}\}$ generate $\mathbb{Z}^{n}.$ Then there is a unique tiling, up to a congruency, of $\mathbb{Z}^{n}$by $V$ and this tiling is lattice. In particular, for a prime $q \leq 7$, there is a unique tiling, up to a congruency, of $\mathbb{Z}^{q-1}$
by semi-crosses.
\end{corollary}

\begin{remark}
We note that a computer aided proof that there is only one tiling of $\mathbb{Z}^{4}$ by semi-crosses is provided in~\cite{Sz} without relating the result to other tiles of size $5$.
\end{remark}

\noindent We note that in the proof of the above theorem we have not used explicitly the fact that $q$ is a prime. We believe that the property distinguishing tilings by semi-crosses of prime size from tilings by semi-crosses of composite size is that of being cyclic.
 We recall that a tiling $\mathcal{T}=\{V+l;l\in \mathcal{L}\}$ is called cyclic if, for each codeword $l$, 
\begin{equation*}
l\in \mathcal{L} \Rightarrow \langle l \rangle \subset \mathcal{L}\text{ ;}
\end{equation*}
that is, if for any codeword, also all its shifts are codewords.
\bigskip
In this regard, at the very end of the paper we show that:

\begin{claim}
For each prime $q>2$, there is a cyclic tiling of $\mathbb{Z}^{q-1}$ by semi-crosses.
\end{claim}

\begin{proof}
For a primitive element $t$ of the multiplicative group $\mathbb{Z}_{q}^{\ast}$ we define a homomorphism $\phi : \mathbb{Z}^{q-1}\rightarrow \mathbb{Z}_{q}$ by
\begin{equation*}
\phi (e_{i})=t^{i-1}\text{ for }i=1, \dots ,q-1.
\end{equation*}
Then $\mathcal{T}=\{V_{q-1}+l; l\in \mathcal{L}=\ker (\phi )\}$ is a lattice tiling of $\mathbb{Z}^{q-1}$ by semi-crosses. Let $a=(a_{1}, \dots ,a_{q-1})\in \mathcal{L}$. Then
\begin{equation*}
a_{1}t^{0}+a_{2}t^{1}+ \dots +a_{q-1}t^{q-2}\equiv 0 \pmod{q}
\end{equation*}
Multiplying the congruence by $t^{k}$ yields
\begin{equation*}
a_{1}t^{0+k}+a_{2}t^{1+k}+ \dots +a_{q-1}t^{q-2+k} \equiv 0 \pmod{q};
\end{equation*}
that is $\pi ^{k}(a)$, the shift of $a$ by $k$ to the right, is a codeword as well.\bigskip 
\end{proof}

%
%

\end{document}